\documentclass[letterpaper]{article}

\newcommand{\myauthor}{Benjamin Antieau and Kenneth Chan}
\newcommand{\mytitle}{Maximal orders in unramified central simple algebras}
\newcommand{\pdftitle}{\mytitle}

\author{Benjamin Antieau\footnote{Benjamin Antieau was supported by NSF Grant DMS-1358832.}~ and Kenneth Chan}
\title{\mytitle}

\usepackage[pdfstartview=FitH,
            pdfauthor={\myauthor},
            pdftitle={\pdftitle},
            colorlinks,
            linkcolor=reference,
            citecolor=citation,
            urlcolor=e-mail,
            ]{hyperref}

\usepackage{mathrsfs}
\usepackage[mathscr]{euscript}
\usepackage{amsmath}
\usepackage{amscd}
\usepackage{amsbsy}
\usepackage{amssymb}
\usepackage{microtype}
\usepackage{enumerate}
\usepackage[all,cmtip]{xy}
\usepackage{tikz}
\usetikzlibrary{matrix,arrows}
\usepackage{dsfont}
\usepackage[abbrev,lite]{amsrefs}
\usepackage{amsthm}

\usepackage{color}
\definecolor{todo}{rgb}{1,0,0}
\definecolor{conditional}{rgb}{0,1,0}
\definecolor{e-mail}{rgb}{0,.40,.80}
\definecolor{reference}{rgb}{.20,.60,.22}
\definecolor{mrnumber}{rgb}{.80,.40,0}
\definecolor{citation}{rgb}{0,.40,.80}

\DeclareMathOperator{\codim}{codim}
\DeclareMathOperator{\hdim}{hdim}
\DeclareMathOperator{\height}{ht}

\DeclareMathOperator{\depth}{depth}

\newcommand{\we}{\simeq}
\newcommand{\iso}{\cong}

\newcommand{\Gm}{\mathds{G}_{m}}

\renewcommand{\epsilon}{\varepsilon}

\newcommand{\QCoh}{\mathrm{QCoh}}
\newcommand{\Coh}{\mathrm{Coh}}

\DeclareMathOperator{\Spec}{Spec}

\DeclareMathOperator{\Hoh}{H}

\DeclareMathOperator{\Ext}{Ext}

\newcommand{\ShTw}{\mathbf{Tw}}

\newcommand{\StEnd}{\mathscr{E}\mathrm{nd}}

\newcommand{\StHoh}{\mathscr{H}}

\newcommand{\et}{\mathrm{\acute{e}t}}
\newcommand{\fppf}{\mathrm{fppf}}

\DeclareMathOperator{\Br}{Br}

\newcommand{\Mrm}{\mathrm{M}}

\newcommand{\Brm}{\mathrm{B}}

\newcommand{\Fscr}{\mathscr{F}}
\newcommand{\Oscr}{\mathscr{O}}

\newcommand{\Sscr}{\mathscr{S}}

\newcommand{\Zscr}{\mathscr{Z}}
\newcommand{\Ascr}{\mathscr{A}}

\newcommand{\Lscr}{\mathscr{L}}

\newcommand{\Escr}{\mathscr{E}}
\newcommand{\Hscr}{\mathscr{H}}
\newcommand{\Xscr}{\mathscr{X}}
\newcommand{\Gscr}{\mathscr{G}}

\theoremstyle{plain}
\newtheorem{theorem}{Theorem}[section]
\newtheorem{lemma}[theorem]{Lemma}
\newtheorem{proposition}[theorem]{Proposition}

\theoremstyle{definition}
\newtheorem{definition}[theorem]{Definition}
\newtheorem{assumption}[theorem]{Assumption}

\newtheorem{remark}[theorem]{Remark}

\let\oldmarginpar\marginpar
\renewcommand\marginpar[1]{\-\oldmarginpar[\raggedleft\footnotesize #1]%
{\raggedright\footnotesize #1}}

\usepackage{txfonts}

\begin{document}
\maketitle

\begin{abstract}
    \noindent
    Using depth of coherent sheaves on noetherian algebraic stacks,
    we construct non-Azumaya maximal orders in unramified central simple algebras over
    schemes of dimension at least $3$.

    \paragraph{Key Words.}Maximal orders, depth, algebraic stacks.

    \paragraph{Mathematics Subject Classification 2010.}
    \href{http://www.ams.org/mathscinet/msc/msc2010.html?t=16Hxx&btn=Current}{16H10},
    \href{http://www.ams.org/mathscinet/msc/msc2010.html?t=13Cxx&btn=Current}{13C15},
    \href{http://www.ams.org/mathscinet/msc/msc2010.html?t=16Exx&btn=Current}{14D23}.
\end{abstract}

Let $X$ be a regular noetherian integral scheme, and assume that $\dim(X)\leq 2$.
Let $A$ be a central simple algebra over the function field $K$ of $X$ with class
$\alpha\in\Br(X)\subseteq\Br(K)$.
Auslander and Goldman showed in~\cite{auslander-goldman} that every maximal order $\Ascr$ in
$A$ is in fact an Azumaya algebra. What happens in higher dimensions has remained unexplored. The
following theorem provides a stark contrast to their result.  

\begin{theorem}\label{thm:existence}
    Suppose that $X$ is a Japanese integral noetherian scheme with function field $K$
    and a regular point of codimension $\geq 3$.
    Let $\alpha\in\Br(X)$ be a Brauer class, and let $A$ be a central
    simple algebra with Brauer class $\overline{\alpha}\in\Br(K)$. If $\deg(A)\geq 2$,
    then there exist non-Azumaya maximal orders on $X$ in $A$.
\end{theorem}

Recall that an integral locally noetherian scheme $X$ is Japanese
if for every non-empty affine open $\Spec R\subseteq X$
the ring $R$ is Japanese. A noetherian domain $R$ with function field $K$ is Japanese if for every finite
field extension $K\subseteq L$ the integral closure $S$ of $R$ in $L$ is a finitely
generated $R$-module. 
This condition holds for integral quasi-excellent schemes, and hence for almost
all rings that one encounters in practice. Similarly, if $X$ is quasi-excellent, then the
regular locus of $X$ is open, so that the existence of a regular point of codimension at
least $3$ can be determined by the codimension of the non-regular locus.

We use Yu's result~\cite{yu} that every order over a Japanese scheme is contained in a
maximal order, while the regular point hypothesis seems to be the easiest way to
ensure that the endomorphism algebra of a non-locally free coherent sheaf is not Azumaya.

To prove the theorem, we will briefly develop some notions of depth and reflexivity for coherent
sheaves on an algebra stack. These are straightforward generalizations of depth and
reflexivity for schemes, but we do not know of a reference for what we need. An alternate
route would be to take a more algebraic approach and describe depth for modules over Azumaya
algebras. In any case, we prove a local depth criterion for reflexivity on algebraic stacks.
Once in hand, this criterion will let us check that certain $\Xscr$-twisted coherent sheaves
of global dimension $1$ are in fact reflexive. Taking the endomorphism algebras of these
modules produces the desired orders.

More specifically, we can assume that $\alpha$ is represented by an Azumaya algebra $\Ascr$
that is a maximal order in $A$, for otherwise there already exists a non-Azumaya maximal order.
We can also assume that $X$ is the spectrum of a regular local ring of dimension at least
$3$. Indeed, any maximal order we construct over such a local ring will extend to a maximal
order over the entire scheme under the hypothesis of the theorem.
Then, $\Ascr\iso\StEnd(\Escr)$, where $\Escr$ is an
$\Xscr$-twisted locally free sheaf of rank $n>1$ and $\Xscr$ is a $\Gm$-gerbe 
representing $\alpha$. For general
$f,g\in\Ascr$, we prove that the $\Xscr$-twisted sheaf $\Fscr$ with
presentation
\begin{equation*}
    0\rightarrow\Escr\xrightarrow{\begin{pmatrix}f\\g\end{pmatrix}}\Escr^2\rightarrow\Fscr\rightarrow 0
\end{equation*}
is reflexive but not locally free. Indeed, since $\Ascr$ is noncommutative we can take $f,g$
locally noncommuting homomorphisms, but where the vanishing locus of $(f,g)$ has codimension $3$ in $X$.
Then, $\StEnd(\Fscr)$ is an example of a non-Azumaya maximal order.

This paper arose out of a desire to better understand the examples constructed by
Antieau and Williams in~\cite{aw4}. They gave an example of a $6$-dimensional smooth
complex affine variety $X$ and a Brauer class $\alpha\in\Br(X)$ with the following
properties: the class $\alpha$ is represented by a degree $2$ division algebra $D$ over the
function field $K$, and no maximal order in $D$ over $X$ is Azumaya. The methods
in~\cite{aw4} are topological in nature, and this paper is a first step in attempting to
understand how to construct purely algebraic examples and to answer the question of whether
this phenomenon can occur in dimension $3$.

\begin{remark}
    The theorem is false if $\deg(A)=1$. Indeed, if $X$ is normal and $\Ascr$ is a maximal order in $K$,
    then $\Ascr$ is in particular reflexive, and hence normal
    by~\cite{hartshorne-reflexive}*{Proposition~1.6}. This means that if $U\subseteq X$ is
    an open subset with $\mathrm{codim}_X(X-U)\geq 2$, then $\Ascr(X)\rightarrow\Ascr(U)$ is
    an isomorphism. Since there is such a $U$ with $\Ascr_U$ Azumaya, we find that
    $\Oscr_X(U)\rightarrow\Ascr(U)$ is an isomorphism for this choice of $U$, which implies that
    $\Oscr_X\rightarrow\Ascr$ is an isomorphism.
\end{remark}

In the final section of the paper, we return to the Auslander-Goldman result mentioned
above, namely that all maximal orders in unramified central simple algebras on regular
$2$-dimensional schemes are Azumaya. We show that this property in fact characterizes
regular integral $2$-dimensional schemes.

\begin{theorem}\label{thm:surfaces}
    Let $X$ be a $2$-dimensional integral noetherian surface with field of fractions $K$.
    Then, $X$ is regular if and only if every maximal order
    over $X$ in a central simple $K$-algebra with unramified Brauer class is Azumaya.
\end{theorem}

\paragraph{Acknowledgements.} We thank Max Lieblich for showing us the second example in
Section~\ref{examples}.

\section{\label{examples}Two examples}

The starting points of our investigation were the following two examples, showing that, at
least in certain cases, Theorem~\ref{thm:existence} holds for regular noetherian schemes.

Consider the example of~\cite{aw4}. There is a smooth affine
complex scheme $X$ with $\dim(X)=6$ and a Brauer class $\alpha\in\Br(X)$ such that there are
no Azumaya maximal orders in the degree $2$ division algebra over $K$ representing $\alpha$.
On the other hand, by Yu~\cite{yu}, there are maximal orders $\Ascr$ in $D$,
because $X$ is normal and noetherian. By construction, these are not Azumaya. The
non-Azumaya locus of such an order $\Ascr$ in $X$ is closed and has codimension at least $3$. We can
localize at a closed point in the non-Azumaya locus to obtain examples over regular noetherian
local rings of dimension $6$.

A more geometric example was explained to us by Max Lieblich. Let $S$ be a smooth
projective surface over an algebraically closed field $k$, and let $\Sscr\rightarrow S$ be a
$\mu_n$-gerbe, where $n$ is prime to the characteristic of $k$. The moduli space
$\ShTw(n,L,c)$ of semi-stable torsion-free $\Sscr$-twisted sheaves of rank $n$, determinant
$L$, and second Chern class $c$ is proper, and the open locus of locally free sheaves inside is
properly contained in $\ShTw(n,L,c)$ for $c$ sufficiently large. It follows from
the valuative criterion that there is a
discrete valuation ring $R$ and a map $\Spec R\rightarrow\ShTw(n,L,c)$ sending to the
generic point to the locally free locus and the closed point to the boundary. This
classifies a torsion-free but non-locally free $\Sscr$-twisted sheaf $\Escr$ on $X_R$, which
one can check is reflexive using depth. Taking the endomorphism algebra $\StEnd(\Escr)$ yields a
maximal order over the $3$-dimensional scheme $X_R$, which is non-Azumaya by~\cite{auslander-goldman-maximal}*{Theorem~4.4}.

\section{Local cohomology and depth on an algebraic stack}

We begin by briefly recalling some preliminaries on algebraic stacks, and then we prove that the
depth criterion for reflexivity~\cite{hartshorne-reflexive}*{Proposition 1.6} holds in this setting. A good reference is the
book of Laumon and Moret-Bailly~\cite{laumon-moret-bailly} or, for another account,
see~\cite{stacks-project}. The following definitions are either standard, or are obvious extensions 
of the scheme-theoretic definitions pertaining to local cohomology which can be found, for example, in~\cite{hartshorne-local}.

\begin{definition}
    An \emph{algebraic stack} over a base scheme $S$ is a stack admitting a \emph{smooth atlas} $p:U\rightarrow\Xscr$,
    where $U$ is a scheme and $p$ is representable (in algebraic spaces), smooth and
    surjective, such that the diagonal morphism
    $\Xscr\rightarrow\Xscr\times_S\Xscr$ is representable.
\end{definition}

\begin{definition}
    We say that $\Xscr$ is \emph{locally noetherian} if it has a smooth atlas
    $p:U\rightarrow\Xscr$ where $U$ is locally noetherian.
\end{definition}

Given an algebraic stack
$\Xscr$, we study sheaves of $\Oscr_\Xscr$-modules, which are by definition sheaves of
$\Oscr_\Xscr$-modules on the associated ringed site $(\Xscr_{\mathrm{fppf}},\Oscr_\Xscr)$.

\begin{definition}
    An $\Oscr_\Xscr$-module $\Fscr$ is \emph{quasi-coherent} if for every $f:\Spec R\rightarrow\Xscr$ the
    pullback $f^*\Fscr$ is quasi-coherent. As usual, it is equivalent to ask for $p^*\Fscr$
    to be quasi-coherent where $p:U\rightarrow\Xscr$ is a smooth atlas.
    If $\Xscr$ is locally noetherian, an $\Oscr_\Xscr$-module $\Fscr$ is \emph{coherent}
    if $p^*\Fscr$ is coherent for some (and hence every) locally noetherian smooth atlas
    $p:U\rightarrow X$.
\end{definition}

\begin{lemma}[\cite{stacks-project}*{\href{http://stacks.math.columbia.edu/tag/0781}{Tag 0781}}]
    The abelian category $\QCoh(\Xscr)$ has enough injectives.
\end{lemma}

Now we can give the definitions of local cohomology and depth.

\begin{definition}
    Given a closed substack $\Zscr\subseteq\Xscr$ and a quasi-coherent sheaf on $\Xscr$,
    we define $\StHoh^0_{\Zscr}(\Fscr)$, the \emph{sheaf of sections with support in
    $\Zscr$}, as the quasi-coherent sheaf
    \begin{equation*}
        (f:\Spec R\rightarrow\Xscr)\mapsto\Hoh^0_{\Spec R\times_{\Xscr}\Zscr}(f^*\Fscr).
    \end{equation*}
    The functor $\StHoh^0_{\Zscr}:\QCoh(\Xscr)\rightarrow\QCoh(\Xscr)$ is left-exact.
\end{definition}

\begin{definition}
    The \emph{local cohomology functors with supports in $\Zscr$} are $\StHoh^i(-)$, the right
    derived functors of $\StHoh^0_{\Zscr}(-)$. These can also be defined by sheafifying the
    local cohomology functors restricted to affine schemes.
\end{definition}

\begin{definition}
    The \emph{depth} of $\Fscr$ along $\Zscr$ is defined to be
    \begin{equation*}
        \depth_{\Zscr}\Fscr=\max\{n:\text{$\StHoh^i_{\Zscr}(\Fscr)=0$ for $i<n$}\}.
    \end{equation*}
    We will only apply this definition to coherent sheaves.
\end{definition}

It follows immediately from the definitions that we can compute depth on an atlas for
$\Xscr$. The goal is eventually to relate depth to reflexivity, to introduce another
notion of depth in the special case of gerbes, and to show that this secondary notion agrees
with the definition just given.

\begin{definition}
    Let $\Fscr$ be a coherent sheaf on a locally noetherian algebraic stack $\Xscr$. We say that $\Fscr$
    is \emph{reflexive} if the natural map $\Fscr\rightarrow\Fscr^{\vee\vee}$ is an isomorphism.
\end{definition}

Just as the depth can be computed on an atlas $p:U\rightarrow\Xscr$, reflexivity can also be
checked on $U$.

\begin{definition}
    \begin{enumerate}
        \item   An algebraic stack is \emph{irreducible} if there is a smooth atlas
            $p:U\rightarrow X$ with $U$ a disjoint union of integral schemes.
        \item   An algebraic stack $\Xscr$ is \emph{reduced} if $U$ is reduced for some (and
            hence every) smooth atlas $p:U\rightarrow\Xscr$.
        \item   An algebraic stack is \emph{integral} if it is reduced and irreducible.
        \item   An algebraic stack is \emph{normal} if $U$ is normal for some (and hence every)
            smooth atlas $p:U\rightarrow\Xscr$.
    \end{enumerate}
\end{definition}

The definition of an irreducible stack is somewhat touchy. For example, another definition
could be that the space of points of $\Xscr$ with the Zariski topology, as defined
in~\cite{laumon-moret-bailly}*{Chapter~5}, is irreducible. However, it is probably not the
case that such a stack admits a smooth atlas $p:U\rightarrow\Xscr$ with $U$ irreducible. The
definition we give suffices for the applications we have in mind below.

\begin{definition}
    A quasi-coherent sheaf $\Fscr$ on an algebraic stack is \emph{torsion-free} if
    $p^*\Fscr$ is torsion free for some (and hence every) smooth atlas
    $p:U\rightarrow\Xscr$.
\end{definition}

Recall that for a scheme $U$, a quasi-coherent sheaf $\Gscr$ is torsion-free if the stalk
$\Gscr_x$ is a torsion free $\Oscr_{U,x}$-module for each point of $U$. If $\Fscr$ is any
coherent sheaf on a locally noetherian algebraic stack, then the dual sheaf $\Fscr^\vee$ is
torsion-free. Indeed, since this is true on noetherian local rings, it is true on $\Xscr$.

\begin{lemma}\label{lem:torsionfree}
    If $\Xscr$ is locally noetherian and reduced, then
    a coherent sheaf $\Fscr$ is torsion-free if and only if $\Fscr\rightarrow\Fscr^{\vee\vee}$ is injective.
\end{lemma}

\begin{proof}
    As subsheaves of torsion-free sheaves are torsion-free, and since $\Fscr^{\vee\vee}$ is
    torsion-free, we see that the condition is sufficient. So, suppose that $\Fscr$ is
    torsion-free. Then, $p^*\Fscr$ is torsion-free for some smooth atlas
    $p:U\rightarrow\Xscr$. Since $U$ is reduced, the canonical map $p^*\Fscr\rightarrow
    p^*\Fscr^{\vee\vee}$ is an isomorphism when restricted to the scheme of generic points
    of $U$. In particular, the kernel of this map is a torsion submodule of $p^*\Fscr$.
\end{proof}

\begin{proposition}[Hartshorne~\cite{hartshorne-reflexive}*{Proposition~1.6}]\label{prop:stackyhart}
    Let $\Fscr$ be a coherent sheaf on a normal integral locally noetherian algebraic stack $\Xscr$. Then, $\Fscr$ is
    reflexive if and only if it is torsion-free and $\Hscr^1_{\Zscr}(\Fscr)=0$ for all closed substacks
    $\Zscr\subseteq\Xscr$ with $\codim_{\Xscr}\Zscr\geq 2$.
\end{proposition}

\begin{proof}
    If $\Fscr$ is reflexive, it is torsion-free by Lemma~\ref{lem:torsionfree}. Hence, $\Hscr^0_\Zscr(\Fscr)=0$ for all
    proper closed substacks $\Zscr\subseteq\Xscr$. Let $p:U\rightarrow\Xscr$ be a smooth
    atlas where $U$ is a disjoint union of integral (normal, locally noetherian) schemes.
    By definition, $\Zscr\subseteq\Xscr$ has $\codim_\Xscr\Zscr\geq 2$, if $\codim_U\Zscr_U\geq 2$, where
    $\Zscr_U=U\times_\Xscr\Zscr$. Hence, since $p^*\Fscr$ is reflexive, the schematic
    version of the present proposition~\cite{hartshorne-reflexive}*{Proposition 1.3} implies that
    $p^*\Hscr^1_\Zscr(\Fscr)\iso\Hscr^1_{\Zscr_U}(p^*\Fscr)=0$. As $p$ is faithfully flat,
    this shows that $\Hscr^1_{\Zscr}(\Fscr)=0$.

    Now, suppose that $\Fscr$ is torsion-free. If $\Fscr$ is not reflexive, then the
    cokernel $\Gscr$ of $\Fscr\rightarrow\Fscr^{\vee\vee}$ is a non-zero coherent sheaf on
    $\Xscr$. Since all torsion-free sheaves on a normal scheme are locally free in
    codimension $1$, it follows that the support of $\Gscr$ is a closed substack $\Zscr$
    codimension at least $2$. The long exact sequence in local cohomology yields
    \begin{equation*}
        0\rightarrow\Hscr^0_{\Zscr}(\Gscr)\rightarrow\Hscr^1_{\Zscr}(\Fscr)\rightarrow\Hscr^1_{\Zscr}(\Fscr^{\vee\vee}),
    \end{equation*}
    since $\Hscr^0_{\Zscr}(\Fscr)=\Hscr^0_{\Zscr}(\Fscr^{\vee\vee})=0$. As
    $\Hscr^0_{\Zscr}(\Gscr)=\Gscr$, this shows that $\Hscr^1_{\Zscr}(\Fscr)\neq 0$. Now, it
    follows by applying $p^*$ that $\Fscr$ is not reflexive.
\end{proof}

\section{Reflexivity on gerbes}

We specialize to the case that $\Xscr\rightarrow X$ is an $A$-gerbe where $A$ is a
smooth affine commutative group scheme, and we fix a character $\chi:A\rightarrow\Gm$. In
particular, the natural map $\Hoh^2_{\et}(X,A)\rightarrow\Hoh^2_{\mathrm{fppf}}(X,A)$ is an
isomorphism. It follows that there is an \'etale cover $U\rightarrow X$ and a
section $p:U\rightarrow\Xscr$, which is a smooth atlas. Recall that when $\Xscr\rightarrow
X$ has a section, there is non-canonical equivalence $\Xscr\we\Brm A$, and $\Brm A\cong [X/A]$.

\begin{assumption}
    In this section $X=\Spec R$ is a normal integral noetherian affine scheme, $I\subseteq
    R$ is a proper ideal, and $\Xscr\rightarrow X$ is an $A$-gerbe where $A$ is a smooth
    affine commutative group scheme with a fixed character $\chi:A\rightarrow\Gm$.
\end{assumption}

In this case an $A$-gerbe $\Xscr\rightarrow X$ as above is a normal integral locally
noetherian algebraic stack.
The abelian category $\QCoh(\Xscr_{\fppf})$ is $R$-linear, so we can give an alternate
definition of depth in this case, which we show reduces to the definition in the previous
section.

\begin{definition}
    \begin{enumerate}
        \item We say that $r\in I$ is a \emph{non-zero divisor} on $\Fscr$ if $\ker(\Fscr \xrightarrow{r} \Fscr)=0$.
        \item   A coherent sheaf $\Fscr$ on $\Xscr$ is $R$-torsion-free if
            $r:\Fscr\rightarrow\Fscr$ is injective for all $0\neq r\in R$.
        \item   An $\Fscr$-regular sequence in $I$ is a sequence of
            elements $x_1,\ldots,x_d$ of $I$ such that
            $x_i:\Fscr/(x_1,\ldots,x_{i-1})\Fscr\rightarrow\Fscr/(x_1,\ldots,x_{i-1})\Fscr$ is injective for
            $1\leq i\leq d$.
        \item   The $I$-depth of $\Fscr$ is the maximal length of an $\Fscr$-regular
            sequence in $I$; we denote this integer by $\depth_I\Fscr$.
    \end{enumerate}
\end{definition}

Note that since $R$ is integral, the definition of torsion-free given here is equivalent to
the more standard definition that asks for the stalks $\Fscr_x$ to be torsion-free
$\Oscr_{X,x}$-modules for all points $x$ of $X$.

The next lemma is an exact analogue of a standard fact about modules over commutative rings.

\begin{lemma}
    If $r\in I$ is a non-zero divisor on $\Fscr$, then $\depth_I\Fscr/r=\depth_I\Fscr-1$.
\end{lemma}

\begin{proof}
    Define $\Hscr^0_I(\Fscr)=\cap_{r\in I}\ker(\Fscr\xrightarrow{r}\Fscr)$, and let
    $\Hscr^n_I(\Fscr)$ be the right derived functors of $\Hscr^0_I(\Fscr)$. Note that
    $\depth_I\Fscr=0$ if and only if $\Hscr^0_I(\Fscr)\neq 0$. If $\Hscr^0_I(\Fscr)\neq 0$,
    then by definition every element of $I$ is a zero-divisor on $\Fscr$, whence
    $\depth_I\Fscr=0$. On the other hand, if $\depth_I\Fscr=0$, then
    \begin{equation*}
        I\subseteq\bigcup_{r\in I}\mathrm{ann}\left(\ker(\Fscr\xrightarrow{r}\Fscr)\right).
    \end{equation*}
    It follows that $I\subseteq\mathrm{ann}\left(\ker(\Fscr\xrightarrow{r}\Fscr)\right)$ for
    some $r\in I$. Since $r$ is a zero-divisor on $\Fscr$, it follows that $\Hscr^0_I(\Fscr)\neq 0$.

    Now, we claim that, just as for finitely generated modules over noetherian commutative
    rings, we have $\depth_I\Fscr\geq d$ if and only if $\Hscr^i_I(\Fscr)=0$ for $i<d$. The
    previous argument proves this for $d=1$. If $\Hscr^i_I(\Fscr)=0$ for $i<d$, then for any
    $0\neq r\in I$ we have $\Hscr^i_I(\Fscr/r)=0$ for $i<d-1$ from the long exact sequence in
    local cohomology. It follows inductively that $\depth_I\Fscr\geq d$.

    We are reduced to proving the following. Suppose that $\depth_I\Fscr=d+1$, and assume
    that for all coherent $\Oscr_\Xscr$-modules $\Gscr$ and all $i\leq d$ we have
    $\depth_I\Gscr\geq i$ if and only if $\Hscr^j_I(\Gscr)=0$ for $0\leq j<i$. Then,
    $\Hscr^d_I(\Fscr)=0$. Suppose that $\Hscr^d_I(\Fscr)$ is non-zero. Since this sheaf
    is $I$-torsion, the kernel of multiplication by $r$ is non-zero for any $\Fscr$-regular
    element of $I$. In particular, if $r$ is part of an $\Fscr$-regular sequence of length
    at least $d+1$, then we see that $\Fscr/r$ satisfies $\depth_I\Fscr/r\geq d$, while
    $\Hscr^{d-1}_I(\Fscr/r)\neq 0$. This contradicts the assumptions. The lemma now follows
    from the long exact sequence in local cohomology.
\end{proof}

\begin{lemma}\label{lem:depths}
    Suppose that $U=\Spec S\rightarrow X$ is an \'etale cover with a section
    $p:U\rightarrow\Xscr$. If $\Zscr=\Spec R/I\times_X\Xscr$, then
    $\depth_I\Fscr=\depth_{\Zscr}\Fscr=\depth_{IS}p^*\Fscr$.
\end{lemma}

\begin{proof}
    Since $\depth_{\Zscr}\Fscr$ is computed using the local cohomology sheaves, and as $p$
    is faithfully flat, it follows
    that $\depth_{\Zscr}\Fscr=\depth_{IS}p^*\Fscr$. So, we will
    prove by induction on $d=\depth_I\Fscr$ that $\depth_I\Fscr=\depth_{IS}p^*\Fscr$.

    If $d=0$, so that $\Fscr$ is $I$-torsion, we have that $\Hscr^0_I(\Fscr)\rightarrow\Fscr$
    is an isomorphism. But,
    then by faithful flatness, we have that $p^*\Hscr^0_I(\Fscr)\rightarrow p^*\Fscr$ is an
    isomorphism. But, $p^*\Hscr^0_I(\Fscr)\iso\Hscr^0_{IS}(p^*\Fscr)$. That is, $\depth_{IS}p^*\Fscr=0$.

    So, assume that the lemma is true for all coherent sheaves on $\Xscr$ with depth at most
    $d$, and assume that $\depth_I\Fscr=d+1$. Let $r$ be a non-zero divisor on $\Fscr$ in
    $I$. Then, $\depth_I\Fscr/r=\depth_{IS}p^*(\Fscr/r)=\depth_{IS}(p^*\Fscr)/r$. The lemma
    now follows from the fact that $\depth_I\Fscr=1+\depth_I\Fscr/r$ and
    $\depth_{IS}(p^*\Fscr)/r=1+\depth_{IS}p^*\Fscr$.
\end{proof}

Putting this all together, we prove the following proposition.

\begin{proposition}
    Let $\Xscr\rightarrow X$ be an $A$-gerbe on a normal integral noetherian affine scheme
    $X=\Spec R$, and let $\chi:A\rightarrow\Gm$ be a character.
    Then, a coherent $\chi$-twisted $\Oscr_\Xscr$-module $\Fscr$ is reflexive if and only if it is
    torsion-free and $\depth_P\Fscr\geq 2$ for all prime ideals $P$ such that
    $\height P\geq 2$.
\end{proposition}

\begin{proof}
    The necessity follows immediately from Lemma~\ref{lem:depths} since $p^*\Fscr$ is
    reflexive for any smooth atlas $p:U=\Spec S \rightarrow\Xscr$. Suppose that $\Fscr$
    is torsion-free and $\depth_P\Fscr\geq 2$ for all primes $P$ with $\height P\geq 2$.
    We let $\Gscr$ be the cokernel of the injective map
    $\Fscr\rightarrow\Fscr^{\vee\vee}$; it is another $\chi$-twisted coherent sheaf, and $\Gscr$ has support consisting of primes of
    height at least $2$. For these primes $P$, we can use the faithful flatness of $R_P\rightarrow
    S_{PS}$ to argue that $\Hscr^1_P(\Fscr)\neq 0$. Indeed, $S_{PS}$ is a semi-local ring
    faithfully flat and \'etale over $R_P$. As $\Gscr_P\neq 0$, it follows that $q^*\Gscr_P$ is non-zero, where
    $q:\Xscr_S\rightarrow\Xscr$. But, the section $p:U\rightarrow\Xscr$ induces a map
    $r:U\rightarrow\Xscr_S$ that induces an equivalence
    $r^*:\QCoh^\chi(\Xscr_S)\rightarrow\QCoh(U)$. Since $p=r\circ q$, it follows that
    $p^*\Gscr_P$ is a non-zero coherent sheaf on $S_{PS}$. Therefore, for some maximal ideal $Q$
    of $S_{PS}$, which necessarily satisfies $\height Q\geq 2$ by the going-down theorem for
    flat extensions~\cite{matsumura}*{Theorem 9.5}, we
    have $(p^*\Gscr_P)_Q\neq 0$. Hence, $\Fscr$ is not reflexive, by
    Proposition~\ref{prop:stackyhart}.
\end{proof}

Using the proposition, we can prove a twisted form of the Auslander-Buchsbaum
formula.

\begin{definition}
    Let $X=\Spec R$ be an affine scheme, $\Xscr\rightarrow X$ an $A$-gerbe where $A$ is a
    smooth affine $X$-group scheme, and $\chi:A\rightarrow\Gm$ a character. Then, a
    $\chi$-twisted quasi-coherent sheaf $\Fscr$ has \emph{homological dimension $\leq n$} if
    $\Ext^i(\Fscr,\Gscr)=0$ for all $i>n$ and all quasi-coherent $\chi$-twisted sheaves $\Gscr$.
    Write $\hdim\Fscr$ for the homological dimension of $\Fscr$, the
    smallest $n$ such that $\Fscr$ has homological dimension $\leq n$.
\end{definition}

\begin{theorem}\label{thm:tab}
    Suppose that $\Xscr\rightarrow X$ is an $A$-gerbe where $X=\Spec R$ is the spectrum of
    an integral noetherian local ring with maximal ideal $M$.
    If $\Fscr$ is a $\chi$-twisted sheaf with finite homological dimension, then
    \begin{equation*}
        \hdim\Fscr+\depth_M\Fscr=\depth_M R.
    \end{equation*}
\end{theorem}

\begin{proof}
    This follows immediately by using an \'etale splitting $\Spec S\rightarrow\Spec R$ for
    the gerbe $\Xscr$. All three numbers are stable under faithfully flat \'etale maps, and
    over $\Spec S$ there is an equivalence of categories $\QCoh^\chi(\Xscr_S)\iso\QCoh(\Spec S)$.
\end{proof}

\section{The proof}

%
%

We prove the main theorem of the paper.

\begin{proof}[Proof of Theorem~\ref{thm:existence}]
    By Yu's result, any order on $X$ is contained in a maximal order. 
    By assumption, there exists a regular point $p$ of codimension $\ge 3$.
    Therefore, if we construct a
    non-Azumaya maximal order in $A$ over $\Spec\Oscr_{X,p}$ and extend it to $X$, it is
    contained in a maximal order that is not Azumaya at $p$. Thus, we now assume that
    $X=\Spec R$ is a regular local ring of dimension at least $3$ with field of fractions $K$ and that $A$ is a
    central simple algebra with unramified Brauer class $\alpha\in\Br(X)\subseteq\Br(K)$.

    We distinguish two cases. If $\alpha=0$ in $\Br(K)$, then we need to construct a non-Azumaya
    maximal order in the matrix algebras $\Mrm_n(K)$ for $n>1$. If $R$ is exactly
    $3$-dimensional, then the first syzygy of a minimal free resolution of $R/M$, where $M$
    is the maximal ideal, is a non-locally free reflexive $R$-module of rank $2$. If $R$ is
    of dimension more than $3$, then one can extend a syzygy such as the one above from a
    $3$-dimensional localization. The upshot is that if $R$ is a regular local ring of
    dimension at least $3$, then there are reflexive but not locally free $R$-modules of any
    rank more than $1$. Taking the endomorphisms of these we gain non-Azumaya maximal orders
    in $\Mrm_n(K)$ for all $n>1$ by~\cite{auslander-goldman-maximal}*{Theorem 4.4} using the
    fact that the Azumaya locus of a maximal order in an unramified central simple algebra
    is the locally free locus~\cite{auslander-goldman}*{Theorem~2.1, Proposition~4.6}.

    If $\alpha\in\Br(K)$ is non-zero, then by Wedderburn's
    theorem we can assume that $A$ is a division algebra of degree at least $2$. Indeed,
    given a non-Azumaya maximal order $\Ascr$ in a division algebra $A$, $\Mrm_n(\Ascr)$ is
    a non-Azumaya maximal order in $\Mrm_n(A)$.
    Wedderburn's theorem tells us that we can assume that $A$ is a division algebra.
    Moreover, by work of Panin on purity~\cite{panin-purity}, we know that there is an Azumaya maximal order
    $\Ascr$ in $A$ over $X$.

%

    Let $g:\Xscr\rightarrow X$ be a $\Gm$-gerbe with obstruction class $\alpha$ (for
    background on $\Gm$-gerbes and $\Xscr$-twisted sheaves, see Lieblich~\cite{lieblich}). It is the
    gerbe of trivializations of $\Ascr$.
    There is a locally free $\Xscr$-twisted sheaf $\Fscr$ such that $\StEnd(\Fscr)\iso
    g^*\Ascr$.
    We will construct a non-locally free $\Xscr$-twisted sheaf $\Escr$ of the same rank as
    $\Fscr$. Consider elements $f,g\in\Ascr$, and assume that they are not both
    zero. We define $\Escr$ as the cokernel
    \begin{equation}
        \label{defEscr}
        0\rightarrow\Fscr\xrightarrow{\begin{pmatrix}f\\g\end{pmatrix}}\Fscr^2\rightarrow\Escr\rightarrow 0.
    \end{equation}
    Since $f,g$ are not both zero, and since $A$ is a division ring,
    the map $\Fscr\rightarrow\Fscr^2$ is injective.
    Now, we search for satisfiable conditions on the pair $f,g$ that ensures that
    $\Escr$ is reflexive but not locally free. The latter is easy: it suffices to assume
    that we cannot solve $af+bg=1$ for $a,b$ in $\Ascr$. In other words, since $A$
    is a division algebra, we assume that $(f,g)$ is contained in the maximal ideal
    $M\Ascr$. In this case, $\Escr$ has homological dimension $1$.

    Now, $\Escr$ is torsion-free if and only if the two-sided ideal $(f,g)$ is not contained in
    $P\Ascr$ for any height $1$ prime $P$ of $R$. To prove this, note that by the
    snake lemma there is an exact sequence
    \begin{equation*}
        0\rightarrow\ker(r:\Escr\rightarrow\Escr)\rightarrow\Fscr/r\xrightarrow{\begin{pmatrix}f\\g\end{pmatrix}}\Fscr^2/r
    \end{equation*}
    for any $r\in R$. Thus, if $(f,g)\subseteq(r)$, the kernel is non-zero. On the other
    hand, if $(f,g)$ is not contained in $(r)$ for some irreducible $r$, then, say, $f$ is a non-zero
    section of $\Ascr/(r)$, and it follows that
    $\ker(r:\Escr\rightarrow\Escr)=0$ since the reduced norm $\mathrm{Nrd}(f)$ does not vanish identically
    along $\Spec R/r\subseteq X$.

    Finally, suppose that $\height P\geq 2$. We must ensure that
    $\depth_P\Escr_P\geq 2$. If $\height P>2$, then this holds by the twisted
    Auslander-Buchschbaum formula of Theorem~\ref{thm:tab} since $\hdim\Escr_P\leq 1$.
    In other words, we must ensure that $\Escr_P$ is locally free for all height $2$ primes
    $P$. This occurs if and only if $(f,g)$ is not contained
    in $P\Ascr$ for a height $2$ prime of $R$. In other words, at least one of $f$
    or $g$ needs to be a unit in $\Ascr_P$ for all height $2$ primes $P$.

    As $\deg(A)\geq 2$, there are two non-commuting units $x,y$ of $\Ascr$. Let
    $t_1,t_2,t_3$ be elements of $R$ such that $I=(t_1,t_2,t_3)$ has codimension at least
    $3$. Then for $f=t_1x+t_2$ and $g=t_3y$, the $\Xscr$-twisted sheaf $\Escr$ defined in 
    (\ref{defEscr}) is reflexive and not locally free. The maximal order $\StEnd(\Escr)$ is then
    non-Azumaya. Indeed, we can take an \'etale cover $Y \to X$ which splits $\alpha$.
    Then $f^{\ast}\StEnd_X(\Escr) \cong \StEnd_{Y}(f^{\ast}\Escr)$. Since $f$ is 
    faithfully flat, $f^{\ast}\Escr$ is reflexive (or locally free) if and only if $\Escr$ is too.
    Since $\alpha$ is trivial on $Y$, there is an $\alpha^{-1}$-twisted line bundle $\Lscr$ such that 
    tensoring by this line bundle gives an equivalence of categories
    $\Coh^{\alpha}(Y\times_X\Xscr)\to \Coh(Y)$.
    Then $\StEnd_{Y}(f^{\ast}\Escr)\cong \StEnd_{Y}(f^{\ast}\Escr\otimes \Lscr)$. 
    By~\cite{auslander-goldman-maximal}*{Theorem~4.4}, the latter
    is not locally free, since $f^{\ast}\Escr\otimes \Lscr$ is reflexive but not locally free.
    Hence $\StEnd_X(\Escr)$ is not locally free.
\end{proof}

The underlying reason for our ability to construct these examples is that the
\emph{vanishing locus} of a non-central section $x$ of an Azumaya algebra $\Ascr$ can be smaller than a
hypersurface. For instance, in the notation above, the vanishing locus of $f$ has
codimension $2$.

\section{Surfaces}

Now we prove our converse to the result of Auslander and Goldman on maximal orders on
regular surfaces.

\begin{proof}[Proof of Theorem~\ref{thm:surfaces}]
Assume that $X$ is regular. Any maximal order $\Ascr$ is reflexive, hence is locally free since
$X$ is regular of dimension $2$. Therefore, $\Ascr$ is Azumaya exactly where it is
unramified by~\cite{auslander-goldman}*{Proposition~4.6}.

Assume that $X$ is not regular, we will construct a non-Azumaya maximal order
in an unramified central simple algebra. To begin, we can assume that $X=
\Spec R$, where $R$ is a $2$-dimensional noetherian local domain.

If $R$ is not normal, let $R \rightarrow S$ be the integral closure of $R$ in
$K$. Then, $S$ is a maximal order in $K$ over $\Spec R$. As normalization is
never flat if it is non-trivial, it follows that $S$ is not locally free over
$\Spec R$, and hence not Azumaya.

Now, assume that $R$ is in addition normal. Since $\dim  X=2$, we can assume
$X$ has isolated singularities so the singular locus of $X$ is the closed
point. By a theorem of Buchweitz~\cite{buchweitz-thesis}, we have $\mathrm{D}_{\mathrm{sg}} ( X )
\cong \underline{\mathrm{MCM}} ( R )$. The triangulated category $\mathrm{D}_{\mathrm{sg}} ( X )$ is the Verdier
quotient $\mathrm{D}^{b} ( \mathrm{Coh} ( X ) ) / \mathrm{Perf} ( X )$ where
$\mathrm{Perf} ( X )$ is the full subcategory of perfect complexes. Hence
$\mathrm{D}_{\mathrm{sg}} ( X )$ is trivial if and only if $X$ is regular. The
category $\underline{\mathrm{MCM}} ( R )$ is a triangulated category whose
objects are maximal Cohen-Macaulay $R$-modules and morphisms are $R$-module
morphisms modulo those which factor through a projective module. 
 Since $R$ is not regular, there exists a
non-projective maximal Cohen-Macaulay $R$-module $M$. Let $\Ascr= \mathrm{End}_{R} (
M )$. Now $\dim  X=2$, so $M$ is reflexive as an $R$-module. By 
~\cite[Proposition 4.1]{auslander-goldman-maximal}, 
$\Ascr$ is reflexive. Moreover, $M$ is locally free in
codimension $1$, hence $\Ascr$ is maximal in codimension $1$. So $\Ascr$ is a maximal
order. The order $\Ascr$ is in $\mathrm{End}_{K} ( V,V )$ where $V=M \otimes_{R} K$,
and hence is unramified.

The maximal order $\Ascr$ might be Azumaya. To produce a maximal order which is
not Azumaya, consider the order
\begin{eqnarray*}
  \Ascr' & = & \mathrm{End}_{R} ( R \oplus M ) \simeq \left(\begin{array}{cc}
    R & M^{\ast}\\
    M & \Ascr
  \end{array}\right) .
\end{eqnarray*}
Again $\Ascr'$ is unramified since $\Ascr'$ is contained in
$\mathrm{End}_{K} ( K \oplus V )$. It is reflexive since $\Ascr,M,M^{\ast}$ are all
reflexive as $R$-modules. Finally, $M,M^{\ast}$ are free in codimension $1$,
hence $\Ascr'$ is maximal in codimension $1$. This shows that $\Ascr'$ is again a maximal
order. However, $\Ascr'$ is not Azumaya since $M$ is not locally free. 

\end{proof}

%
%
%
%
%
%


\begin{bibdiv}
\begin{biblist}

\bib{aw4}{article}{
    author = {Antieau, Benjamin},
    author = {Williams, Ben},
    title = {On the non-existence of Azumaya maximal orders},
    journal = {Invent. Math.},
    volume = {197},
    number = {1},
    pages = {47--56},
    year = {2014},
}
\bib{auslander-goldman-maximal}{article}{
    author={Auslander, Maurice},
    author={Goldman, Oscar},
    title={Maximal orders},
    journal={Trans. Amer. Math. Soc.},
    volume={97},
    date={1960},
    pages={1--24},
    issn={0002-9947},
}

\bib{auslander-goldman}{article}{
author={Auslander, Maurice},
author={Goldman, Oscar},
title={The Brauer group of a commutative ring},
journal={Trans. Amer. Math. Soc.},
volume={97},
date={1960},
pages={367--409},
issn={0002-9947},
}
%
%
%
%

\bib{buchweitz-thesis}{article}{
    author={Buchweitz, Ragnar-Olaf},
    title={Maximal Cohen-Macaulay Modules and Tate-Cohomology Over Gorenstein Rings},
    eprint={https://tspace.library.utoronto.ca/handle/1807/16682},
    publisher = {University of Hannover}
    date={1986},
}


\bib{hartshorne-local}{book} {
	AUTHOR = {Hartshorne, Robin},
	TITLE = {Local cohomology},
	SERIES = {A seminar given by A. Grothendieck, Harvard University, Fall},
	VOLUME = {1961},
	PUBLISHER = {Springer-Verlag, Berlin-New York},
	YEAR = {1967},
	PAGES = {vi+106},
}


\bib{hartshorne-reflexive}{article}{
    author={Hartshorne, Robin},
    title={Stable reflexive sheaves},
    journal={Math. Ann.},
    volume={254},
    date={1980},
    number={2},
    pages={121--176},
    issn={0025-5831},
}


\bib{laumon-moret-bailly}{book}{
    author={Laumon, G{\'e}rard},
    author={Moret-Bailly, Laurent},
    title={Champs alg\'ebriques},
    series={Ergebnisse der Mathematik und ihrer Grenzgebiete. 3. Folge. A
        Series of Modern Surveys in Mathematics},
    volume={39},
    publisher={Springer-Verlag},
    place={Berlin},
    date={2000},
    pages={xii+208},
    isbn={3-540-65761-4},
    review={\MR{1771927 (2001f:14006)}},
}
\bib{lieblich}{article}{
    author={Lieblich, Max},
    title={Twisted sheaves and the period-index problem},
    journal={Compos. Math.},
    volume={144},
    date={2008},
    number={1},
    pages={1--31},
    issn={0010-437X},
}

\bib{matsumura}{book}{
    author={Matsumura, Hideyuki},
    title={Commutative ring theory},
    series={Cambridge Studies in Advanced Mathematics},
    volume={8},
    edition={2},
    note={Translated from the Japanese by M. Reid},
    publisher={Cambridge University Press},
    place={Cambridge},
    date={1989},
    pages={xiv+320},
    isbn={0-521-36764-6},
}



\bib{panin-purity}{article}{
    author={Panin, I. A.},
    title={Purity conjecture for reductive groups},
    journal={Vestnik St. Petersburg Univ. Math.},
    volume={43},
    date={2010},
    number={1},
    pages={44--48},
    issn={1063-4541},
}


\bib{stacks-project}{article}{
    author       = {The {Stacks Project Authors}},
    title        = {\itshape Stacks Project},
    eprint = {http://stacks.math.columbia.edu},
    year         = {2014},
}

%
%



\bib{yu}{article}{
    author={Yu, Chia-Fu},
    title={On the existence of maximal orders},
    journal={Int. J. Number Theory},
    volume={7},
    date={2011},
    number={8},
    pages={2091--2114},
    issn={1793-0421},
}




\end{biblist}
\end{bibdiv}
\end{document}